\documentclass{article}
\usepackage[affil-it]{authblk}
\usepackage{graphicx}
\usepackage[space]{grffile}
\usepackage{latexsym}
\usepackage{amsfonts,amsmath,amssymb}
\usepackage{url}
\usepackage[utf8]{inputenc}
\usepackage{hyperref}
\hypersetup{colorlinks=false,pdfborder={0 0 0}}
\usepackage{textcomp}
\usepackage{longtable}
\usepackage{multirow,booktabs}

\usepackage[a4paper,left=35mm,right=35mm,top=30mm,bottom=30mm,marginpar=25mm]{geometry}
\usepackage{amsmath}
\usepackage{amssymb}
\usepackage{amsthm}
\usepackage{epsfig}
\usepackage{hyperref}
\usepackage{dsfont}
\usepackage[displaymath,mathlines]{lineno}
\usepackage{comment}

\renewcommand{\div}{\operatorname{div}}

\newcommand{\Rr}{{\mathbb{R}}}

\newcommand{\Nn}{{\mathbb{N}}}
\newcommand{\Zz}{{\mathbb{Z}}}

\newcommand{\Tt}{{\mathbb{T}}}

\def\dx{{\rm d}x}

\def\leq{\leqslant}
\def\geq{\geqslant}

\allowdisplaybreaks
\makeindex
\numberwithin{equation}{section}
\newtheoremstyle{thmlemcorr}{10pt}{10pt}{\itshape}{}{\bfseries}{.}{10pt}{{\thmname{#1}\thmnumber{
#2}\thmnote{ (#3)}}}
\newtheoremstyle{thmlemcorr*}{10pt}{10pt}{\itshape}{}{\bfseries}{.}\newline{{\thmname{#1}\thmnumber{
#2}\thmnote{ (#3)}}}
\newtheoremstyle{defi}{10pt}{10pt}{\itshape}{}{\bfseries}{.}{10pt}{{\thmname{#1}\thmnumber{
#2}\thmnote{ (#3)}}}
\newtheoremstyle{remexample}{10pt}{10pt}{}{}{\bfseries}{.}{10pt}{{\thmname{#1}\thmnumber{
#2}\thmnote{ (#3)}}}
\newtheoremstyle{ass}{10pt}{10pt}{}{}{\bfseries}{.}{10pt}{{\thmname{#1}\thmnumber{
A#2}\thmnote{ (#3)}}}
\theoremstyle{thmlemcorr}
\newtheorem{theorem}{Theorem}
\numberwithin{theorem}{section}
\newtheorem{lemma}[theorem]{Lemma}
\newtheorem{corollary}[theorem]{Corollary}
\newtheorem{proposition}[theorem]{Proposition}

\theoremstyle{thmlemcorr*}
\newtheorem{theorem*}{Theorem}
\newtheorem{lemma*}[theorem]{Lemma}
\newtheorem{corollary*}[theorem]{Corollary}
\newtheorem{proposition*}[theorem]{Proposition}
\newtheorem{problem*}[theorem]{Problem}
\newtheorem{conjecture*}[theorem]{Conjecture}
\theoremstyle{defi}

\newtheorem{hyp}{Assumption}

\newtheorem{problem}{Problem}
\theoremstyle{remexample}
\newtheorem{remark}[theorem]{Remark}

\theoremstyle{ass}

\usepackage{color}

\begin{document}

\title{Existence of positive solutions for an approximation of stationary
mean-field games}

\author{Nojood Almayouf}
\affil{Effat University, Kingdom of Saudi Arabia}

\author{Elena Bachini}
\affil{University of Padova, Italy}

\author{Andreia Chapouto}
\affil{University of Coimbra, Portugal}

\author{Rita Ferreira}
\affil{KAUST, Kingdom of Saudi Arabia}

\author{Diogo Gomes}
\affil{KAUST, Kingdom of Saudi Arabia}

\author{Daniela Jord\~ao}
\affil{University of Coimbra, Portugal}

\author{David Evangelista Junior}
\affil{KAUST, Kingdom of Saudi Arabia}

\author{Avetik Karagulyan}
\affil{Yerevan State University, Armenia}
  
\author{Juan Monasterio}
\affil{University of Buenos Aires, Argentina}

\author{Levon Nurbekyan}
\affil{KAUST, Kingdom of Saudi Arabia}

\author{Giorgia Pagliar}
\affil{University of Verona, Italy}

\author{Marco Piccirilli}
\affil{University of Rome Tor Vergata, Italy}

\author{Sagar Pratapsi}
\affil{University of Coimbra, Portugal}

\author{Mariana Prazeres}
\affil{KAUST, Kingdom of Saudi Arabia}

\author{Jo{\~{a}}o Reis}
\affil{University of Lisbon, Portugal}

\author{Andr\'e Rodrigues}
\affil{University of Coimbra, Portugal}

\author{Orlando Romero}
\affil{University of Lisbon, Portugal}

\author{Maria Sargsyan}
\affil{American University of Armenia, Armenia}

\author{Tommaso Seneci}
\affil{University of Verona, Italy}

\author{Chuliang Song}
\affil{Zhejiang University, China}

\author{Kengo Terai}
\affil{Waseda University, Japan}

\author{Ryota Tomisaki}
\affil{Waseda University, Japan}

\author{Hector Velasco-Perez}
\affil{National Autonomous University
of Mexico, Mexico}
  
\author{Vardan Voskanyan}
\affil{KAUST, Kingdom of Saudi Arabia} 
  
\author{Xianjin Yang}
\affil{KAUST, Kingdom of Saudi Arabia}

\date{\today}

\bibliographystyle{plain}

\maketitle 

\begin{abstract}
Here, we consider a regularized mean-field game model that features a
low-order regularization. We prove the existence of solutions with
positive density. To do so, we combine a priori estimates with the
continuation method. In contrast with high-order
regularizations, the low-order regularizations are easier to
implement numerically. Moreover, our methods give a theoretical foundation for
this approach.

\end{abstract}

%\title[The short title]{Template}
%
%\address[D. A. Gomes]{
%       King Abdullah University of Science and Technology (KAUST), CEMSE Division , Thuwal 23955-6900. Saudi Arabia, and  
%       KAUST SRI, Center for Uncertainty Quantification in Computational Science and Engineering.}
%\email{diogo.gomes@kaust.edu.sa}

%\address[J. Doe]{
%       The University Address}
%\email{doe@doe.tv}

%\keywords{Mean Field Game; OTHER KEWYWORDS}
%\subjclass[2010]{
%       35J47, %Second order elliptic systems
%       35A01} %Existence problems: global existence, local existence, non-existence

\thanks{
Rita  Ferreira, Diogo Gomes, David Evangelista Junior,
Levon Nurbekyan, Mariana Prazeres,  Vardan Voskanyan, and Xianjin Yang
were partially supported by KAUST baseline and start-up funds and 
KAUST SRI, Uncertainty Quantification Center in Computational Science and Engineering.  The other authors were partially supported by KAUST Visiting Students Research Program.
} 

\section{Prologue}

On August 22, 2015, eighteen young mathematicians
(B.Sc.\! and M.Sc.\! Students) arrived at  King Abdullah
University of Science and Technology (KAUST) in
Thuwal, Kingdom of Saudi Arabia. They were  participants
in
the first KAUST summer camp in Applied Partial Differential
Equations. Among them were Argentinians, Armenians,
Chinese, Italians, Japanese, Mexicans, Portuguese,
and Saudis. For many of them, this was their first
time abroad. All were looking forward to the following
three weeks.

We designed the summer camp to give an intense
hands-on three-week Ph.D. experience. It comprised
courses, seminars, a project, and a final presentation.
The project was an essential component of the
summer camp, and its main outcome is the present
paper. Our objectives were to introduce
students to an active research topic, teach effective
paper writing techniques, and develop their presentation
skills. Numerous challenges
lay ahead. First, we had three weeks to achieve
these goals. Second, students had distinct backgrounds.
Third, we planned to study a
research-level problem, not a simple exercise.

We selected a problem in mean-field games, a recent
and active area of research.
The primary goal was to prove the existence of
solutions of a system of partial differential
equations.
To avoid unnecessary technicalities, we considered
the one-dimensional case, where the partial differential
equations become ordinary differential equations.
The project involved partial differential equation
methods that are usually taught in advanced courses:
a priori estimate methods, the infinite dimensional
implicit function theorem, and the continuation
method.  In spite
of the elementary nature of the proofs, the results
presented here are a relevant and original contribution
to the theory of mean-field games.

We divided the students into five groups and assigned
tasks to each of them. Roughly, each of the sections
of this paper corresponds to a task.
The students were given a rough statement of
the results to be proven, and their task was to figure out
the appropriate assumptions, the precise statements, and the proofs. 
The work of the different groups had to be coordinated to make sure that
the  assumptions, results, and proofs fit nicely with each other
and that duplicate work was avoided.  
Several KAUST\  graduate students and post-docs
were of invaluable help in this  regard.

This project would not have been possible within such a short time frame without the use of new technologies. 
The paper was written in a
collaborative fashion using the platform \url{http://authorea.com}
that allowed all the groups to work simultaneously. In this way, all  groups had access to the latest version of the assumptions and to the current statements of the theorems and propositions. Each
group  could easily comment and make corrections on other group's work.

This project illustrates how research in mathematics can be a collaborative experience even with a large number of participants. Moreover, it gave  each of the students in the summer camp a glimpse of real research in mathematics.
Finally, this was the first experience
for the Ph.D. students and post-docs who helped in this project in mentoring and advising students. This summer camp was a unique and valuable experience
for all participants whose results we share in this paper.

\section{Introduction}

Mean-field game (MFG) theory is the study of strategic decision making in 
large populations of small interacting individuals
who are also called agents or players.   
The MFG framework was developed in the engineering community by Caines, Huang, and Malham\'e \cite{Caines2,
Caines1} and in the mathematical community by Lasry and Lions \cite{ll1,ll2,ll3} (also see \cite{LCDF}). These games model the behavior of rational agents who play symmetric differential games.
In these problems, each player chooses their optimal strategy in view of  global (or macroscopic) statistical information on the ensemble of players.
This approach leads to novel problems in nonlinear equations.
%that contain as particular examples many classical problems.
%and are linked to several research fields of Analysis
Current research topics are the applications of
MFGs (including, for example, growth theory in economics and environmental policy), mathematical problems related to MFGs
(existence, uniqueness, and regularity questions),
 and numerical methods in the MFGs framework (discretization, convergence, and efficient implementation).

Here, we consider the following problem:
\begin{problem}
        \label{P1}
Let $\Tt=\Rr/ \Zz$ denote the one-dimensional torus, identified with the interval $[0,1]$ whenever convenient. 
Fix a $C^2$ Hamiltonian, $H:\Rr\to \Rr$, and a continuous potential, $V:\Tt\to \Rr$. Let $\alpha$ and $\epsilon$ 
be positive numbers with \(\epsilon \leq 1\)
for definedness. 
Find $u,m\in C^2(\Tt)$ satisfying $m>0$ and
\begin{equation}
\label{PP1}
\begin{cases}
u-u_{xx}+H(u_x)+V(x)=m^\alpha+\epsilon (m-m_{xx})
\\
m-m_{xx}-(H'(u_x)m)_x=1-\epsilon (u-u_{xx}).
\end{cases}
\end{equation}
\end{problem}
In this problem,  
$m$ is the distribution of players and $u(x)$ is the value function for a typical player in the state $x$. We stress that the condition $m>0$ is an essential component of the problem. So, if $(u,m)$ solves the Problem \ref{P1}, we require $m$ to be strictly positive.
We will show the existence of solutions to this problem under 
suitable assumptions on the Hamiltonian that are described in Section \ref{assp}. An example that satisfies those assumptions is $H(p)= (1+p^2)^{\gamma/2}$
with $1<\gamma<2$, and any $V:\Tt\to \Rr$ of class $C^2$. 
  
When $\epsilon=0$, \eqref{PP1} becomes
\begin{equation}
\label{PP2}
\begin{cases}
u-u_{xx}+H(u_x)+V(x)=m^\alpha
\\m-m_{xx}-(H'(u_x)m)_x=1.
\end{cases}
\end{equation}
The  system in \eqref{PP2}  is
a typical MFG model similar to the one introduced in \cite{ll1}.
The Legendre transform of the Hamiltonian, $H$,
given by $L(v)=\sup_{p} -p v-H(p)$ is the cost in units of time that an 
agent incurs by choosing to move with a drift $v$;
the potential, $V$, accounts for spatial preferences of the agents; the term $m^\alpha$ encodes congestion effects.

The MFG models proposed in \cite{ll1,ll2} consist of a system of partial differential equations that have \eqref{PP2} as a particular case.
The current literature covers a broad range of problems, including stationary problems  \cite{GPatVrt,
GPM1, GR,GM,  PV15}, heterogeneous populations
\cite{MR3333058}, 
time-dependent models \cite{cgbt, GPim1, GPim2,  GPM3, GPM2,porretta, porretta2}, congestion problems \cite{GMit, Graber2}, and 
 obstacle-type problems \cite{GPat}. For a recent account of the theory of MFG, we suggest the survey paper \cite{GS} and the courses
 \cite{LCDF} and \cite{LIMA}. 

The system in \eqref{PP1} arises as an approximation of \eqref{PP2} 
that preserves monotonicity properties. Monotonicity-preserving approximations 
to MFG systems were introduced in \cite{FG2}. In that paper, 
the authors consider  mean-field games in dimension $d\geq 1$ that include
the following example:
\begin{equation}
\label{FGS}
\begin{cases}
u-\Delta u+H(Du,x)+V(x)=m^\alpha+\epsilon (m+\Delta ^{2q} m)+\beta_\epsilon(m)
\\
m-\Delta m-\div (D_pH(Du,x)m)=1-\epsilon (u+\Delta ^{2q}u),
\end{cases}
\end{equation}
where $q$ is a large enough integer, and $\beta_\epsilon$ is a suitable penalization that satisfies $\beta_\epsilon(m)\to -\infty$ as $m\to 0$. 
Then, as $\epsilon\to 0$, the solutions of \eqref{FGS} converge to solutions of \eqref{PP2}. 
Yet, from the perspective of numerical methods, both the high-order degree of \eqref{FGS} and the singularity caused
by the penalty, $\beta_\epsilon$, are unsatisfactory due to a poor conditioning of discretizations. Here, we investigate a low-order regularization
that may be more suitable for computational problems. 

A fundamental difficulty in the analysis of \eqref{PP1} is the non-negativity of $m$. 
The Fokker-Planck equation in \eqref{PP2} has a maximum principle, and, consequently, $m\geq 0$ for any solution of \eqref{PP2}. 
Due to the coupling, this property is not evident in the corresponding equation in \eqref{PP1}. 
The previous regularization in \eqref{FGS} relies on a penalty that forces the positivity of $m$. 
This mechanism does not exist in \eqref{PP1}, and 
we are not aware of any general method to prove the existence of positive solutions
of \eqref{PP1}. 
  
Our main result is the following theorem:
\begin{theorem}
\label{T1}
Suppose Assumptions~\ref{ass1}--\ref{ass6} hold
(cf. Section~\ref{assp}) . Then, there exists $\epsilon_0>0$ such that for all \( 0< \epsilon < \epsilon_0 \),
Problem \ref{P1} admits a $C^{2, \frac 1 2}$ solution $(u,m)$.
\end{theorem}

Theorem~\ref{T1} introduces 
a low-order regularization procedure for \eqref{PP2} for which existence of solutions can be established without penalty terms. 
Because high-order regularization methods and penalty terms create serious difficulties in the numerical implementation, 
this result is relevant to the numerical approximation of \eqref{PP2}.
Moreover, we believe that the techniques 
we consider here can be extended to higher-dimensional problems. 

To prove the main result, we use the continuation method. The first step is to establish a priori estimates for the solutions of \eqref{PP1}. Then, we replace the potential, $V$, by $\lambda V$ for $0\leq \lambda\leq 1$. For $\lambda=0$, which corresponds to $V=0$ in \eqref{PP1}, we determine an explicit solution. The a priori estimates give that the set, $\Lambda$, of values, $\lambda$, for which \eqref{PP1} has a solution is a closed set. Finally, we apply an
infinite-dimensional version of the implicit function theorem to show that $\Lambda$ is relatively open in $[0,1]$. This
proves the existence of solutions.

The remainder of this paper is structured as follows.
We discuss the main assumptions in 
Section \ref{assp}. Next, 
in Section \ref{partcase},
we start our study of \eqref{PP1}
by considering the case $V=0$ and constructing an explicit solution.
 Sections \ref{enes}--\ref{lowboum} are devoted to a priori estimates 
 for solutions of \eqref{PP1}. These estimates include energy and second-order bounds, discussed respectively 
 in Sections  \ref{enes} and \ref{second}, H\"older and $C^{2,\frac 1 2}$ estimates, addressed respectively in 
 Sections \ref{hold} and \ref{hr}, and lower bounds on $m$,  given in Section \ref{lowboum}.
Next, we lay out the main results needed for the implicit function theorem. 
We introduce the linearized operator in Section \ref{linop} and
discuss
its injectivity and surjectivity properties. 
Finally, the proof of 
Theorem \ref{T1}
is presented in Section \ref{pthm}.

% In section \ref{enes} we present the energy estimates of the system, and then in section \ref{secordes} some second-order estimates.
%Properties of H\"older continuity and of high-regularity of the solution are proven, respectively, in section \ref{hold} and \ref{hr}. Then, in section \ref{boundux} a bound on %the first derivative of the solution $u$ is shown, while in section \ref{lowboum} a lower bound of $m$ is discussed.
%In section \ref{linop} the linearized operator is described, and its properties of injectivity and surjectivity are presented in the next two sections.
%Some applications and examples are discussed in section \ref{eaa}.

%We end this introduction with an outline of this paper. The assumptions are discussed in section \ref{assp}, where also we present an outline of the paper. 
%The proof of the theorem is presented in section 
%\ref{pthm}. Some applications and examples are discussed in section \ref{eaa}.

\section{Main Assumptions}
\label{assp}

To prove Theorem \ref{T1}, we need to introduce various assumptions that are natural in this class of problems. These
encode distinct properties of the Hamiltonian in a convenient way.
We begin by stating a polynomial growth condition for the Hamiltonian.

\begin{hyp}
\label{ass1}
There exist 
positive
 constants, $C_1, C_2, C_3$, and $\gamma>1$, such that for all \(p\in\Rr\),
the Hamiltonian, $H$, satisfies
\[
-C_1+C_2|p|^\gamma\leq H(p)\leq C_1+C_3|p|^\gamma.
\]      
\end{hyp}

For convex Hamiltonians, the expression $pH'(p)-H(p)$ is the Lagrangian written in momentum coordinates. 
The next assumption imposes polynomial growth in this quantity. 

\begin{hyp}
\label{ass2}
There exist positive
constants, $\tilde C_1$, $\tilde C_2$, and  $\tilde C_3$, such that for all \(p\in\Rr\),
we have
\[
-\tilde C_1+\tilde C_2|p|^\gamma\leq pH'(p)-H(p)\leq \tilde C_1+\tilde C_3|p|^\gamma.
\]      
\end{hyp}

%\begin{remark}
%The previous assumptions is most likely overly restrictive, maybe only one inequality is necessary.
%\end{remark}
Because we look for solutions $(u,m)\in C^{2,\frac 1 2}(\Tt)\times  C^{2,\frac 1 2}(\Tt)$
of Problem \ref{P1},
we require in  Assumption~\ref{ass2.5} and  Assumption~
\ref{ass3.5}
more regularity for $V$ and $H$. 
\begin{hyp}
        \label{ass2.5}
The potential, $V$, is of class $C^2$. 
\end{hyp}

Because the Hamilton-Jacobi equation in \eqref{PP2} arises from an optimal control problem, it is natural to suppose that the
Hamiltonian, $H$, is convex. 

\begin{hyp}\label{H_convex}
$H$ is convex.  
\end{hyp}

\begin{hyp}
        \label{ass3.5}
        The Hamiltonian, $H$, is of class $C^4$. 
\end{hyp}

Here, we work with subquadratic Hamiltonians.
Accordingly, we impose  the following condition
on \(\gamma\).

\begin{hyp}
        \label{ass5}
$\gamma<2$. 
\end{hyp}

Finally, we state a growth condition on the derivative of the Hamiltonian.
The exponent, $\gamma$, is the same as in Assumptions \ref{ass1} and \ref{ass2}.
This is a natural growth condition that the model 
$H(p)=(1+|p|^2)^{\frac \gamma 2}$ satisfies. 
 \begin{hyp}
   \label{ass6}
 There exists a positive  
 constant, $ \bar C$,  such that for all \(p\in\Rr\),
we have
 \[
  |H'(p)|\leq \bar C(1+
|p|^{\gamma-1}).
 \]     
  \end{hyp}

\section{The  $V=0$ case}
\label{partcase}

To prove Theorem \ref{T1}, we use the continuation method. More precisely, we consider  system \eqref{PP1} with $V$ replaced by $\lambda V$ for $0\leq \lambda\leq 1$. 
Next, we  show the existence of the solution for all $0\leq \lambda\leq 1$.
As a starting point, we study the  $\lambda=0$
case; that is, $V=0$. We  show that \eqref{PP1} admits a solution in this particular instance.

  \begin{proposition}\label{prop_1} 
  Suppose that $V=0$. Then, there exists an $\epsilon_0>0$ such that for all $0 < \epsilon < \epsilon_0$,   Problem \ref{P1} admits a  solution $(u,m)$.
%Moreover, \((u,m) \in C^{2, \frac 1 2}(\Tt) \times %C^{2, \frac 1 2}(\Tt;
%  ]0,\infty[) \). 
   \end{proposition}
  \begin{proof}
    We look for constant solutions $(u,m)$. In this case, we have $u_x = u_{xx} = m_x = m_{xx} = 0$. Accordingly, \eqref{PP1} reduces to
    \[
    %\begin{equation}
      \begin{cases}
        u+H(0)=m^\alpha+\epsilon m\\
        m=1-\epsilon u.
      \end{cases}
    %\end{equation}
    \]
    
    In the previous system, solving the first equation for $u$ and replacing the resulting expression into the second, we get
    \begin{equation}
      \label{functiong}
      \epsilon m^\alpha + (1+\epsilon^2)m -1 - \epsilon H(0) = 0.
    \end{equation}
    We set $g(m)=\epsilon m^\alpha + (1+\epsilon^2)m -1 - \epsilon H(0)$, so that \eqref{functiong} reads $g(m)=0$.
    Next, we notice that $g(0) = -1 - \epsilon H(0)$. For small enough $\epsilon_0>0$  and for all $0<\epsilon<\epsilon_0$, we have $g(0) < 0$.
    On the other hand, if we take a constant $C > \lvert H(0) \rvert$, we have
    \[
    %\begin{align}
    g(1+\epsilon C) %&= \epsilon (1+ \epsilon C)^ \alpha + (1+\epsilon^ 2)(1+\epsilon C)-1-\epsilon H(0) > \nonumber \\
    %&
    > 1+\epsilon C - 1 - \epsilon H(0) =\epsilon(C-H(0)) > 0.
    %\label{ineq4}
    %\end{align}
    \]
    
    Because $0 < 1+\epsilon C$, by the intermediate value theorem, there exists  a constant $m_0\in ]0,1+\epsilon C[$ such that $g(m_0)=0$. Then, setting $u_0=(1-m_0)/\epsilon$, we conclude that the pair $(u_0,m_0)$ satisfies the requirements. 
  \end{proof}
  
  \begin{remark}
    Note that if $H(0) > 0$, then $g(0) < 0$ and $g(1+\epsilon C)>0$. In this case, the previous proposition holds for all $\epsilon>0$.
  \end{remark}

\section{Energy estimates}
\label{enes}

MFG systems such as \eqref{PP2} admit many a priori estimates. Among those, energy estimates stand out for their elementary proof -- the multiplier method.  
Here, we apply this method to \eqref{PP1}.
\begin{proposition} \label{enint}
Suppose that Assumptions     
\ref{ass1} and \ref{ass2} hold. 
Let $(u,m)$ solve Problem~\ref{P1}. 
Then,  
\begin{equation}
\label{foest}
\int_0^1 m^{\alpha + 1} \,dx+ \int_0^1 |u_x|^{\gamma}(1+m) \,dx+ \epsilon \int_0^1 \left(u^2+m^2+u_x^2+m_x^2 \right) dx \leq C,
\end{equation}
where $C$ is a universal positive constant
depending only on the constants in Assumptions~\ref{ass1} and \ref{ass2}  and on $\|V\|_{L^\infty}$.
\end{proposition}

\begin{proof}
We begin by multiplying the first equation in
\eqref{PP1} by
$(1+\epsilon - m)$ and the second one by $u$.
Adding the resulting expressions and integrating, we get
\begin{multline}
\label{equation1}
\int_0^1 \left[ (1+\epsilon)H(u_x) + m(u_x H'(u_x)
- H(u_x)) \right] dx + \int_0^1 m^{\alpha + 1}
\, dx + \epsilon
\int_0^1 (u^2 + m^2 + u^2_x + m^2_x)\, dx \\
= -\epsilon \int_0^1 u\, dx + \int_0^1 (m-1-\epsilon)V(x)
\, dx+ (1 + \epsilon) \int_0^1 m^\alpha\, dx + \epsilon(1+\epsilon)\int_0^1
m\, dx,
\end{multline}
where we also used integration by parts and the periodicity of $u$ and $m$   to obtain
\[
\int_0^1 mu_{xx} \, dx - \int_0^1 um_{xx} \,dx = 0,  
\]
\[
\int_0^1 u_{xx}\, dx = \left. u_x \right|_0^1
= 0, \quad \int_0^1 m_{xx}\, dx = \left. m_x \right|_0^1
= 0,
\]
\[
 \int_0^1 m m_{xx} \, dx
=  - \int_0^1 m^2_x \, dx, \quad \int_0^1 u u_{xx}
\, dx = - \int_0^1 u^2_x \, dx,
\]
and 
\[
\int_0^1 u(H'(u_x)m)_x\, dx = - \int_0^1 u_x  H'(u_x)m \, dx.
\]
Next, we observe that by Assumptions \ref{ass1} and \ref{ass2}, and using the fact that $0<\epsilon\leq 1$, we have 
\begin{equation}
\label{eeqq2}
\begin{aligned}
& \int_0^1 \left[(1+\epsilon)H(u_x) + m(H'(u_x)
u_x - H(u_x)) \right] \,dx  \\ 
&\quad\geq  \int_0^1 \left[ -2C_1 - \tilde{C_1}
m + K_0|u_x|^\gamma \left(1+m
\right) \right] dx,
\end{aligned}
\end{equation}
where \(K_0 := \min\{C_2, \tilde C_2 \}\).

From \eqref{equation1} and \eqref{eeqq2}, it follows
that
\begin{equation}
\label{eeqq33}
\begin{aligned}
&\int_0^1  K_0 |u_x|^\gamma(1+m) \, dx + \int_0^1  m^{\alpha + 1}
\, dx + \epsilon \int_0^1 \left(u^2 + m^2
+ u_x^2 + m_x^2 \right) dx \\
&\quad\leq  \frac \epsilon 2 \int_0^1 u^2\, dx
+ \frac 1 2 +\left( \Vert V\Vert_\infty+2 + \tilde
C_1 \right) \int_0^1 m \,
dx + 2 \int_0^1 m^\alpha dx + 2\left( \Vert V\Vert_\infty
+ C_1 \right),
\end{aligned}
\end{equation}
where we also used the estimates \(2 u \leq u^2 + 1\)
and  $0<\epsilon\leq 1$.

Finally, we observe that for every $\delta_1, \delta_2>
0$, there exist constants,
$K_1$ and $K_2$, such that 
\begin{equation}
\label{eeqq3}
\int_0^1 m^\alpha \,dx \leq 
\delta_1 \int_0^1 m^{\alpha + 1} \,dx +  K_1, \quad
\int_0^1 m\, dx \leq \delta_2 \int_0^1 m^{\alpha+1}\,dx
+ K_2.
\end{equation}
Consequently, taking \(\delta_1 = \frac 1 8\)
and \(\delta_2 = \frac 1 {4(\Vert V\Vert_\infty+2 + \tilde C_1)}\) in \eqref{eeqq3} and using the
resulting estimates in \eqref{eeqq33}, we conclude
that \eqref{foest} holds.
\end{proof}

\begin{corollary}
\label{c_3} 
Suppose that Assumptions~\ref{ass1} and \ref{ass2}
hold. 
Let $(u,m)$ solve Problem~\ref{P1}. Then, 
\[
\int_0^1 m \,dx \leq C,
\]
where $C$ is a universal positive constant
depending only on the constants in Assumptions~\ref{ass1} and \ref{ass2}  and on $\|V\|_{L^\infty}$.
\end{corollary}
\begin{proof}
Due to  \eqref{foest} and because $m$ is positive,
\[
\int_0^1 m^{\alpha + 1} \leq C,
\]
where $C$ is a universal positive constant
depending only on the constants in Assumptions~\ref{ass1}
and \ref{ass2}  and on $\|V\|_{L^\infty}$.
Consequently, using Young's inequality, we have
that
\[
\int_0^1 m\, dx \leq  \frac{1}{\alpha+1}\int_0^1 m^{\alpha + 1} \,dx + \frac{\alpha}{\alpha+1}
\leq   \frac{C}{\alpha+1}+
 \frac{\alpha}{\alpha+1}. \qedhere
 \]
\end{proof}

\section{Second-order estimates}
\label{second}

We proceed in our study of \eqref{PP1} by examining another technique to 
obtain a priori estimates. These estimates give additional control 
over high-order norms of the solutions.

 \begin{proposition}
 \label{soe}
%Suppose Assumptions    
%\ref{ass1} and \ref{ass2} hold. 
Suppose that Assumption~\ref{ass2.5} holds. Let $(u,m)$ solve Problem~\ref{P1}.
Then, we have
\begin{equation}
\label{P61i}
\int_0^1 \left(H''(u_x)u_{xx}^2m+\alpha m^{\alpha-1}m_{x}^2 \right) dx+ \epsilon\int_0^1 \left(m_x^2 +m_{xx}^2+u_x^2 +u_{xx}^2\right) dx\leq C,
\end{equation}
where $C>0$ denotes a universal constant
depending only on  $\|V\|_{C^2}$.
Moreover, under Assumption~\ref{H_convex},
\begin{equation}
\label{P61ii}
\int_0^1 \alpha m^{\alpha-1}m_{x}^2 \, dx+ \epsilon\int_0^1 \left(m_x^2 +m_{xx}^2+u_x^2 +u_{xx}^2\right) dx\leq C.
\end{equation}
\end{proposition}

\begin{proof} 
To simplify the notation, we  represent by
\(C\) any positive constant that depends only
on \(\Vert V\Vert_{C^2}\) and  whose value
may change from one instance to another.

Multiplying the first equation in \eqref{PP1} by $m_{xx}$ and the second one by $u_{xx}$ yields
\begin{align*}
&\big(u-u_{xx}+H(u_x)+V(x)\big)m_{xx}=\big(m^\alpha+\epsilon (m-m_{xx})\big)m_{xx},
\\\notag
&\big(m-m_{xx}-(H'(u_x)m)_x\big) u_{xx}=\big(1-\epsilon (u-u_{xx}\big)u_{xx}.
\end{align*}
Subtracting the above equations  integrated over $[0,1]$ gives
\begin{equation} 
\label{g3_1}
\begin{aligned}
&\int_0^1\big( u m_{xx} - m u_{xx}+ u_{xx}\big)\,dx + \int_0^1\Big(H(u_x)m_{xx} + \big( H'(u_x) m\big)_x  u_{xx}\Big)\,dx 
\\
&\quad+\int_0^1 V(x)m_{xx}\,dx - \int_0^1 m^\alpha m_{xx}\,dx
+\epsilon \int_0^1 \big( -m m_{xx} + m_{xx}^2 - u u_{xx}+
u_{xx}^2 \big)  \,dx =0.
\end{aligned}
\end{equation}
Next, we evaluate each of the integrals above.  Using the integration by parts formula and the periodicity of boundary conditions, we have
\begin{equation}
\label{g3_2}
\int_0^1 \big( u m_{xx} - m u_{xx} +u_{xx} \big)\,
dx=0.
\end{equation}
In addition,
\begin{equation}
\label{g3_3}
\begin{aligned}
&\int_0^1\left[\big( H'(u_x) m\big)_x  u_{xx} +H(u_x) m_{xx} \right] dx\\
&\quad = \int_0^1  \left[H''(u_x) m  u^2_{xx}+ (H(u_x))_x m_x + \big( H(u_x) \big) m_{xx} \right]dx
 =\int_0^1 H''(u_x)mu^2_{xx}\, dx.
\end{aligned}
\end{equation}
Furthermore, we have
\begin{equation}\label{g3_4}
-\int_0^1 m^\alpha m_{xx}\, dx=\int_0^1 \alpha m^{\alpha-1} m^2_{x}\, dx
\end{equation}
and
\begin{equation}\label{g3_5}
\int_0^1 -V m_{xx} \, dx= -\int_0^1 V_{xx} m \, dx\leq 
\int_0^1 |V_{xx}| m \, dx\leq C \int_0^1 m \, dx\leq C,
\end{equation}
where we used Corollary~\ref{c_3}.

Finally,
\begin{equation}\label{g3_6}
\epsilon\int_0^1  \big( -m m_{xx}  + m_{xx}^2 - u u_{xx}+
u_{xx}^2 \big) \,dx=\epsilon \int_0^1 \left( m_x^2 +  m_{xx}^2+  u_x^2  +  u_{xx}^2 \right) \, \dx.
\end{equation}
Using \eqref{g3_1}--\eqref{g3_6}, we get
\begin{align*}
&\int_0^1 H''(u_x)mu^2_{xx} \, \dx + \int_0^1 \alpha m^{\alpha-1} m^2_{x} \, \dx\\
& \quad
+\epsilon \int_0^1 \left( m_x^2
+  m_{xx}^2+  u_x^2  +  u_{xx}^2 \right) \, \dx
=-\int_0^1 V m_{xx}\leq C. 
\end{align*}
This completes the proof of \eqref{P61i}. To conclude
the proof of Proposition~\ref{soe}, we observe
that Assumption~\ref{H_convex} implies that \(H''\)
is a non-negative function, which together with
\eqref{P61i} gives \eqref{P61ii}.
\end{proof}

\section{H\"older continuity}
\label{hold}

We recall that Morrey's theorem in one-dimension \cite{E6} gives the following result.
\begin{proposition} 
\label{morrey}
Let $f \in C^1(\mathbb{T})$. Then,
  \begin{equation*}
    \left \lvert f(x) - f(y) \right \rvert \leq \left \lVert f_x \right \rVert_{L^2}  \left \lvert x - y \right \rvert ^\frac{1}{2},\quad \forall x,y \in \mathbb{T}.
  \end{equation*}
\end{proposition}

\begin{proposition}
\label{hbounds}
Suppose that Assumptions~\ref{ass1}--\ref{H_convex} hold. 
Let $(u,m)$ solve Problem~\ref{P1}. 
Then, $u$, $u_x$, $m$, and $m_x$
are
$\frac 1 2$-H\"older continuous functions 
with $L^\infty$-norms and H\"older constants bounded by $\frac C {\sqrt{\epsilon}}$, 
where $C$ is a universal constant 
depending only on the constants in Assumptions  
\ref{ass1} and \ref{ass2}  and on $\|V\|_{C^2}$. 
\end{proposition}

\begin{proof}
By Proposition \ref{enint},  we have that
\begin{equation} \label{solutions_L2_norms_bounded}
    \epsilon \int_0^1 \left(m^2 + u^2 + m_x^2 + u_x^2 \right)\, dx\leq C,
\end{equation}
where 
 $C$ is a universal constant 
depending only on the constants in Assumptions~\ref{ass1} and \ref{ass2}  and on $\|V\|_{L^\infty}$. 

According to Proposition \ref{morrey}, we have 
\begin{equation}
    \label{iMorrey}
    \left \lvert u(x) - u(y) \right \rvert \leq \left \lVert u_x \right \rVert_{L^2}  \left \lvert x - y \right \rvert ^\frac{1}{2},\quad \forall x,y \in \mathbb{T}.
\end{equation}
Moreover, combining the bound on $\|u\|_{L^2}$
given by \eqref{solutions_L2_norms_bounded}, 
the mean-value theorem for definite integrals, and the H\"older continuity
given by \eqref{iMorrey}, we get the $L^\infty$ bound on $u$. 
A similar inequality holds for $m$. 
Next, we observe that Proposition~\ref{soe} (see
\eqref{P61ii}) gives
bounds for $\|u_{xx}\|_{L^2}$ and $\|m_{xx}\|_{L^2}$
of the same type of \eqref{solutions_L2_norms_bounded}. 
Accordingly, the functions $u_x$ and $m_x$ are also $\frac 1 2 $-H\"older
continuous, and their $L^\infty$ norms are bounded by 
 $\frac C {\sqrt{\epsilon}}$, where $C$ depends only on the constants in 
 Assumptions~\ref{ass1} and \ref{ass2} and on $\|V\|_{C^2}$.
\end{proof}

\begin{remark}
\label{indepV}
Consider Problem~1 with \(V\) replaced by \(\lambda V\) for some
\(\lambda
\in [0,1]\).   By revisiting the proofs of Propositions~\ref{enint}
and \ref{soe}, we can  readily check that
 the bounds stated in
these propositions are uniform with respect to 
\(\lambda
\in [0,1]\). More precisely, \eqref{foest}, \eqref{P61i},
and \eqref{P61ii} are still valid for a universal positive constant, \(C\), that depends only on the constants in Assumptions~\ref{ass1} and \ref{ass2}  and on $\|V\|_{C^2}$.  
  In particular, Proposition~\ref{hbounds}
remains unchanged.

\end{remark}

\section{Higher Regularity}
\label{hr}

The bounds in the previous section give H\"older regularity for any solution $(u,m)$ of Problem~1 and for its derivatives $(u_x,m_x)$. Here, we use
\eqref{PP1} to improve this result and 
prove H\"older regularity for $u_{xx}$ and $m_{xx}$.

\begin{proposition} \label{hrofsol}
Suppose that Assumptions~\ref{ass1}--\ref{ass3.5} hold.
Let $(u,m)$ solve Problem~\ref{P1}.
Then $(u,m) \in C^{2,\frac 1 2} (\Tt) \times C^{2,\frac 1 2} (\Tt)$.  
\end{proposition}

\begin{proof}
Solving for $m-m_{xx}$ in the second equation of \eqref{PP1}
and replacing the resulting expression in the first equation
yields
\begin{equation}
\label{7-1}
[1+\epsilon^2+\epsilon H''(u_x)m]u_{xx}=(1+\epsilon^2)u+H(u_x)-\epsilon +V(x)-m^{\alpha}-\epsilon H'(u_x)m_x.
\end{equation}
Because $H$ is convex, we have $H''(u_x)\geq 0$. Consequently,  $1+\epsilon^2+\epsilon H''(u_x)m\geq 1>0$. This allows us to rewrite \eqref{7-1} as
\begin{equation}
\label{ABC123}
u_{xx}=\frac{(1+\epsilon^2)u+H(u_x)-\epsilon +V(x)-m^{\alpha}-\epsilon H'(u_x)m_x}{1+\epsilon^2+\epsilon H''(u_x)m}.
\end{equation}
Because $u$, $m$, $u_x$, and $m_x$ are $\frac 1 2$-H\"{o}lder continuous and because \(H\)
and \(H'\) are locally Lipschitz functions, it
follows that  \[
(1+\epsilon^2)u+H(u_x)-\epsilon +V(x)-m^{\alpha}-\epsilon H'(u_x)m_x
\]
 is also $\frac
1 2$-H\"{o}lder continuous.
Similarly, due to Assumption \ref{ass3.5}, 
$1+\epsilon^2+\epsilon H''(u_x)m$ is also $\frac 1 2$-H\"{o}lder continuous and bounded from below. Therefore,  $u_{xx}$ is $\frac 1 2$-H\"{o}lder continuous; thus,  $u \in C^{2,\frac 1 2}(\Tt)$.

Finally, we observe that the second equation in \eqref{PP1} is equivalent
to
\begin{equation}
\label{7-2}
m_{xx}=m+\epsilon(u-u_{xx})-1-H''(u_x)mu_{xx}-H'(u_x)m_x.
\end{equation}
Hence, analogous arguments to those used above yield
that   $m_{xx}$ is also $\frac
1 2$-H\"{o}lder continuous. Thus,
 $m \in C^{2,\frac 1 2}(\Tt)$.
 \end{proof}

%\section{Bounds in $u_x$}
%\label{boundux}

%  \begin{proposition}
%  Under \ref{H_convex}, let $\epsilon >0$ and %$(u,m)$ be a %$C^2$ solution of the problem \eqref{PP1} %such that $m>0$. %Then $u_x$ is bounded by
%  \begin{equation}
%  \label{bound_ux}
%  \lvert u_{x}\rvert \leq \sqrt{\frac{C}{\epsilon}}, 
%  \end{equation}
%  where $C>0$ is the universal constant from %equation %\ref{solutions_L2_norms_bounded}, independent %of $\epsilon$. 
%  \end{proposition}
%  
%  \begin{proof}
%  Since $u$ is $C^2$, it follows that $u_x$ is %$C^1$, and %thus, from section \ref{hold} we %know that the %$\frac{1}{2}$-H\"older semi-norm %of $u_x$ is controlled by %$\big(\int u_{xx}^2\big)^{1/2}$. %That is,
%  $$\frac{\lvert u_x(z) - u_x(y) \rvert}{ \lvert %z-y %\rvert^{\small{1/2}}}  \leq \Big(\int u_{xx}^2 %\Big)^{\small{1/2}},$$
%  for every $y,z \in [0,1]$. 
%  
%  From section \ref{secordes}, we also know that
%  $$\epsilon \int u_{xx}^2 \leq \epsilon \int %\big(m_{x}^2 %+m_{xx}^2 +u_{x}^2 + u_{xx}^2 \big) % \leq C,$$
%  
%  Using the periodicity of $u$, we know that %$u(0) = u(1)$, %so there must exist some $\overline{z} %\in (0,1)$ such that %$u_x(\overline{z})=0$.
% 
%  Putting it all together we have,
%  $$\lvert u_x(y)\rvert = \lvert {u_x(\overline{z}) %- u_x(y)} %\rvert \leq \lvert\overline{z}-y\rvert^{1/2} %\| %u_{xx}\|_{L_{2}} \leq \sqrt{\frac{C}{\epsilon}}, %\quad %\forall \, y \in [0,1].$$
%  And since $u$ was an arbitrary solution, this %completes the %proof.
%  \end{proof}

\section{Lower bounds on $m$}
  \label{lowboum}
  
Here, we establish our last a priori estimate,
which gives lower bounds on $m$.  
We begin by proving an auxiliary result.  

\begin{lemma}
\label{lowboum_lemma}
Suppose that Assumptions~\ref{ass1}--\ref{H_convex}, \ref{ass5}, and \ref{ass6} hold. 
Let $(u,m)$ solve Problem~\ref{P1}.
%
%   Suppose $u$ is a solution to Problem \ref{PP1} and assume Assumptions %\ref{ass1}--\ref{ass6}.
   Then,
   there exists $\bar \epsilon_0>0$ such that for all $0<\epsilon<\bar\epsilon_0$, we have 
   $\left\lVert \epsilon(u-u_{xx}) \right\rVert_{\infty} <\frac 1 2$.
\end{lemma}

\begin{proof} 
We  show that
\begin{equation}
\label{limeu-euxx}
\lim_{\epsilon\to0} \left\lVert \epsilon(u-u_{xx}) \right\rVert_{\infty} = 0,
\end{equation}
from which Lemma~\ref{lowboum_lemma} easily follows.

 To simplify the notation, in the remainder of
this proof, \(C\) represents a positive constant that
is independent of \(\epsilon\) and whose value
may change from one instance to another.

By Proposition~\ref{hbounds}, we have that \(\Vert u\Vert_\infty
\leq C/\sqrt\epsilon\). Thus, 
\begin{equation}
\label{limeu}
\lim_{\epsilon\to0} \left\lVert \epsilon u
\right\rVert_{\infty} = 0.
\end{equation}

Next, we examine $\left\lVert \epsilon u_{xx}
\right\rVert_\infty$. The identity \eqref{ABC123} and the condition  $1+\epsilon^2+\epsilon H''(u_x)m > 1$
give
\begin{equation}
\label{esteuxx}
\left\lVert \epsilon u_{xx} \right\rVert_\infty
\leq  \left\lVert(1+\epsilon^2) u\right\rVert_\infty+\left\lVert
\epsilon H(u_x)\right\rVert_\infty+\epsilon^2+\left\lVert
\epsilon V\right\rVert_\infty+\left\lVert \epsilon
m^\alpha\right\rVert_\infty+\left\lVert\epsilon^2
H'(u_x)m_x\right\rVert_\infty.
\end{equation}
By \eqref{limeu} and by the boundedness of \(V\), it follows that
\(
\lim_{\epsilon \to 0} \big( \left\lVert\epsilon(1+\epsilon^2) u\right\rVert_\infty+\epsilon^2+\left\lVert
\epsilon V\right\rVert_\infty \big) =0.
\)

According to Propositions~\ref{enint} and \ref{soe},
we have that
\[
\int_0^1 m^{\alpha+1} \,d x \leq C \quad \text{and}\quad
\int_0^1\alpha m^{\alpha-1}m^2_x \,d x = \frac{4\alpha}{(\alpha+1)^2}
\int_0^1 \left(m^\frac{\alpha+1}{2}\right)_x^2\,d
x\leq C.
\] 
The first integral guarantees that there
exists  $x_0 \in \Tt$ such that $m^{\frac{\alpha +1}{2}}(x_0) \leq
C$. Then, because \(m>0\) and because \(m\in C^1(\Tt)\),
the second integral together with Proposition~\ref{morrey} implies that for all \(x\in \Tt\),
\begin{equation*}
0< m^\alpha(x) = \left( m^\frac{\alpha}{2}(x)\right)^2\leq
\left( m^\frac{\alpha+1}{2}(x) -  m^\frac{\alpha+1}{2}(x_0)
+\  m^\frac{\alpha+1}{2}(x_0) +\ 1\right)^2 \leq
C.
\end{equation*}
Hence, 
\(
\lim_{\epsilon\to0}\left\lVert \epsilon m^\alpha \right\rVert_\infty =0.
\)

Assumption~\ref{ass1} and Proposition \ref{hbounds}
give
 \[\left\lvert H(u_x)\right\rvert \leq C \big(1+
  \epsilon^{-\frac{\gamma}{2}}\big).\]
 This implies that $ \lim_{\epsilon\to0}\left \lVert \epsilon H(u_x)
\right\rVert_\infty = 0$ because $\gamma <2$
according to Assumption \ref{ass5}.

Combining Assumption~\ref{ass6} with Proposition
\ref{hbounds} gives the bound 
\[ \left\lvert H'(u_x) \right\rvert \leq C \big(1+
  \epsilon^{-\frac{\gamma-1}{2}}\big).
\]
By Proposition~\ref{hbounds}, we have that  $\left\lvert
m_x \right\rvert \leq C/\sqrt{\epsilon}$.
Therefore, invoking Assumption~\ref{ass5} once
more, \(
\lim_{\epsilon\to0}\big\Vert \epsilon^2 H'(u_x)m_x
\big\Vert_\infty = 0\).

Collecting all the limits proved above, we conclude
from \eqref{esteuxx} 
that \(\lim_{\epsilon\to0} \left\lVert \epsilon u_{xx}
\right\rVert_\infty =0\). This equality together with \eqref{limeu}
proves \eqref{limeu-euxx}.
 \end{proof}

\begin{proposition}
\label{lboundsprop}
Suppose that Assumptions~\ref{ass1}--\ref{H_convex}, \ref{ass5}, and \ref{ass6} hold. Assume that  $0<\epsilon<\bar\epsilon_0$, where $\bar\epsilon_0$
is determined by Lemma \ref{lowboum_lemma}.
Let $(u,m)$ solve Problem \ref{P1}.
Then, there exists a  constant $\bar m>0$ such that 
 $m>\bar m$ in \(\Tt\).
Moreover, $\bar m$  
is a universal constant 
depending only on  the constants in Assumptions~\ref{ass1}, \ref{ass2}, and \ref{ass6}  and on $\|V\|_{\infty}$.
\end{proposition}

  \begin{proof}
   Multiplying the second equation in \eqref{PP1} by $1/m$ and integrating with respect to $x$ in $[0,1]$, we obtain
\begin{equation}
\label{mult1overm}
 \int_0^1 \left( 1-\frac{m_{xx}}{m}-\frac{(H'(u_x)m)_x}{m} \right)d x = \int_0^1 \left( \frac{1}{m}-\epsilon \frac{u-u_{xx}}{m} \right) d x.
\end{equation}
    Integration by parts and periodicity yields
    \[
      \int_0^1 \frac{m_{xx}}{m} \,d x = \int_0^1 \frac{m_x^2}{m^2} \,d x.
    \]
  Then,  \eqref{mult1overm} can be rewritten as
    \[
      \int_0^1 \left( \frac{1}{m} + \frac{m_{x}^2}{m^2} \right) \,d x = 1 + \int_0^1 \frac{\epsilon(u-u_{xx})}{m} \,d x - \int_0^1 \frac{(H'(u_x)m)_x}{m} \,d x.
     \]
     
     Next, we estimate the right-hand side of this identity.
     By 
    Lemma \ref{lowboum_lemma}, for $0<\epsilon<\bar\epsilon_0$,  we have $\left\lVert \epsilon(u-u_{xx})\right\rVert_\infty <1/2$. Consequently, 
\begin{equation}
\label{bonmaa}
 \int_0^1 \left( \frac{1}{2m} + \frac{m_{x}^2}{m^2} \right )\,d x \leq 1 + \left\lvert \int_0^1 \frac{(H'(u_x)m)_x}{m} \,d x \right\lvert =  1+\left\lvert \int_0^1 H'(u_x)\frac{m_x}{m}\,d x \right\rvert, 
\end{equation}
where in the last equality we use the integration by parts formula and the periodicity of $u_x$.
     In view of Cauchy's inequality, we conclude that
\begin{equation}
\label{bonmbb}
\left\lvert \int_0^1 H'(u_x)\frac{m_x}{m} \,d x \right\rvert \leq \int_0^1 \left\lvert H'(u_x)\frac{m_x}{m} \right\rvert \,d x \leq \int_0^1 \left( \frac{(H'(u_x))^2}{2}  + \frac{m_x^2}{2m^2}\right)d x.
\end{equation}
     Invoking Assumptions~\ref{ass5} and \ref{ass6},
we obtain the estimates 
\begin{equation*}
(H'(u_x))^2 \leq \bar C^2(1+ |u_x|^{\gamma-1})^2
\leq 2 \bar C^2 \big(1+ |u_x|^{2(\gamma-1)}\big)
\leq 2 \bar C^2 \big(2+ |u_x|^2\big)
\end{equation*}
in \(\Tt\). These estimates,
\eqref{bonmaa},  \eqref{bonmbb}, and Proposition~\ref{enint} yield
\[
       \int_0^1 \left( \frac{1}{2m} + \frac{m_{x}^2}{2m^2} \right) d x \leq 1+ \bar C^2 \left(2+ \frac{C}{\epsilon}\right) .
\]
Consequently,
for \(\tilde C = 2+ 2\bar C^2 \left(2+ \frac{C}{\epsilon}\right)\),
we obtain the two following bounds     \[
       \int_0^1 \frac{1}{m} \,d x \leq \tilde
C \quad \text{and}\quad \int_0^1\frac{m_x^2}{m^2}\,d x= \int_0^1 \big(\ln (m)\big)_x^2 \,d x \leq \tilde
C.
     \] 
The first bound implies that there exists $x_0 \in \Tt$ such that \(\frac{1}{m(x_0)} \leq \tilde C +1\); that is, \(\ln(m(x_0)) \geq - \ln(\tilde
C +1)\). The second bound, together with Proposition~\ref{morrey},
implies that for all \(x\in\Tt\), \(|\ln(m(x))
- \ln(m(x_0))| \leq \sqrt{\tilde C}\). Hence,
for all \(x\in\Tt\),
     \[
       m(x) \geq e^{-\sqrt{\tilde C} - \ln(\tilde
C +1)},
     \]
   which completes the proof.
 \end{proof}

 \begin{remark}
 \label{indepV2} 
 As in Remark~\ref{indepV}, the statement
 of Proposition~\ref{lboundsprop} remains unchanged if we
 replace \(V\) by \(\lambda V\) for some \(\lambda\in[0,1]\) in Problem~1. 
  \end{remark}

\section{The linearized operator}
  \label{linop}
  Consider the functional, $F$, defined for $(u,m,\lambda)\in
  C^{2, \frac 1 2}(\Tt) \times C^{2, \frac 1 2}(\Tt;
  ]0,\infty[) \times \left[0,1\right]$
  by
  \begin{equation}
  \label{F}
  F(u, m, \lambda) = 
  \begin{bmatrix}
  u - u_{xx} + H(u_x) + \lambda V - m^\alpha - \epsilon(m-m_{xx}) \\
  m - m_{xx} - \left(H'(u_x)m\right)_x - 1 + \epsilon(u - u_{xx})
  \end{bmatrix}.
  \end{equation}
 Note that under Assumption~\ref{ass3.5}, the functional $F$ is a $C^1$ map between
 $C^{2, \frac 1 2 }(\Tt) \times C^{2, \frac 1 2}(\Tt;
  ]0,\infty[)\times \left[0,1\right]$
 and $C^{0, \frac 1 2 } (\Tt)\times C^{2, \frac 1 2 }(\Tt)$.

To prove Theorem~\ref{T1}, we use the continuation method and show that for every \(\lambda \in [0,1]\),
the equation    
\begin{equation}
\label{F=0}
 F(u,m,\lambda)=0
\end{equation}
has a solution, $(u,m) \in C^{2, \frac 1 2 } (\Tt)\times C^{2, \frac 1 2}(\Tt;
  ]0,\infty[)$.
Theorem~\ref{T1} then follows by taking \(\lambda=1\)
and by observing that  system \eqref{PP1} is equivalent to 
$F(u,m,1)=0$. 

The implicit function theorem plays a crucial
role in proving the solvability of \eqref{F=0}. To use this theorem, for each \(\lambda\in[0,1]\),
we introduce  the linearized operator \(L\) of $F(\cdot,\cdot,\lambda)$ at $(u,m)\in
  C^{2, \frac 1 2}(\Tt) \times C^{2, \frac 1 2}(\Tt;
  ]0,\infty[)$; that is,

 \begin{equation}\label{lin_op}
  \begin{aligned}
  L(f,v) &= 
   \frac{\partial F}{\partial \mu} (u + \mu v, m + \mu f, \lambda) \Big|_{\mu=0} \\
&= 
  \begin{bmatrix}
  v - v_{xx} + H'(u_x) v_x - \alpha m^{\alpha - 1} f - \epsilon(f - f_{xx}) \\
  f - f_{xx} - \left(H''(u_x)v_xm + H'(u_x)f \right)_x + \epsilon(v - v_{xx})
  \end{bmatrix}
  \end{aligned}
  \end{equation}
for \((f,v) \in C^{2, \frac 1 2 } (\Tt) \times
C^{2, \frac 1 2 } (\Tt)\).
  Under Assumption~\ref{ass3.5} and because $(u,m)\in
  C^{2, \frac 1 2}(\Tt) \times C^{2, \frac 1 2}(\Tt;
  ]0,\infty[)$, \(L\) defines a map from \(C^{2, \frac 1 2 } (\Tt) \times
C^{2, \frac 1 2 } (\Tt)\) into  $ C^{0, \frac 1 2 }(\Tt)\times  C^{0, \frac 1 2 }(\Tt)$. Moreover,
 this map is  continuous and linear. 
   Next, we show that it is also an isomorphism between \(C^{2,
\frac 1 2 } (\Tt) \times
C^{2, \frac 1 2 } (\Tt)\) and  $ C^{0, \frac
1 2 }(\Tt)\times  C^{0, \frac 1 2 }(\Tt)$.
 
%
%\section{Injectivity and surjectivity of the linearized operator}
%\label{ABC}

\begin{proposition}
\label{Lisisom}
Suppose that Assumptions~\ref{H_convex} and \ref{ass3.5} hold. 
Fix $\lambda\in[0,1]$
and assume that  $(u,m)\in
  C^{2, \frac 1 2}(\Tt) \times C^{2, \frac 1 2}(\Tt;
  ]0,\infty[)$ satisfies 
$F(u,m,\lambda)=0$. 
Then, the operator, $L$, given by \eqref{lin_op} is an isomorphism between
$C^{2,\frac 1 2 }(\Tt)\times C^{2,\frac 1 2 }(\Tt)$ and $C^{0,\frac 1 2 }(\Tt)\times C^{0,\frac 1 2 } (\Tt)$. 
\end{proposition}
\begin{proof}
To prove the proposition, we begin by applying
the Lax-Milgram theorem in $H^1(\Tt)\times H^1(\Tt)$,
after which we bootstrap additional regularity. Here, we endow \(H^1(\Tt)\times H^1(\Tt)\) with
the inner product
\begin{equation*}
\langle (\theta_1,\theta_2),(\bar \theta_1, \bar
\theta_2)\rangle_{H^1(\Tt)\times H^1(\Tt)} = \int_0^1
\left( \theta_1 \bar\theta_1 +  \theta_2 \bar\theta_2
+  \theta_{1x} \bar\theta_{1x} + \theta_{2x} \bar\theta_{2x}\right)
dx
\end{equation*}
for \( (\theta_1,\theta_2)\), \((\bar \theta_1, \bar
\theta_2)\in H^1(\Tt)\times
H^1(\Tt) \).

Consider the bilinear form $B:(H^1(\Tt)\times H^1(\Tt))\times (H^1(\Tt)\times H^1(\Tt))\rightarrow \mathbb{R}$ defined for \((v,f)\), \((w_1,w_2)
\in H^1(\Tt)\times
H^1(\Tt) \) by
\begin{equation*}
\begin{aligned}
B\left(
\begin{pmatrix}v\\f\end{pmatrix},\begin{pmatrix}w_1\\w_2\end{pmatrix}\right)
=& \int_0^1 \left( f + \epsilon v  \right) w_1\, dx +\int_0^1 \left[f_x
+ H''(u_x) v_x m + H'(u_x)
f + \epsilon v_x\right] w_{1x}\, dx \\
&-\int_0^1 \left[ v + H'(u_x) v_x - \alpha m^{\alpha-1}
f -\epsilon f \right] w_2\, dx  + \int_0^1 \left(\epsilon f_x - v_x\right) w_{2x}\, dx.
\end{aligned}
\end{equation*}
Note that if $(v,f) \in
 C^{2, \frac 1 2 } (\Tt) \times
C^{2, \frac 1 2 } (\Tt)$, then
\begin{equation}
\label{Bvfreg}
  B\left(
\begin{pmatrix}v\\f\end{pmatrix},\begin{pmatrix}w_1\\w_2\end{pmatrix}\right) =\int_0^1{\left[-L_1(f,v)w_2+L_2(f,v)w_1\right]
dx,}
\end{equation}
 where $L_1$ and $L_2$ are
the first and second components of $L$, respectively. 

Next, we prove that \(B\) is coercive and bounded
in
\(H^1(\Tt)\times
H^1(\Tt)\).  Fix \( (v,f)\), \((w_1,
w_2)\in H^1(\Tt)\times
H^1(\Tt) \).
Using the integration by parts formula and the periodicity of $v$ and $f$, we obtain
\begin{equation*}
B\left(\begin{pmatrix}v\\f\end{pmatrix},\begin{pmatrix}v\\ f \end{pmatrix}\right)=
\int_0^1 \left[{\alpha m^{\alpha-1}f^2+H''(u_x)v_x^2m+
\epsilon(v^2+v_x^2+f^2+f_x^2)}\right]
dx.
\end{equation*}
Because $H''\geq0$ by Assumption
\ref{H_convex} and because  $m>0$, 
we have
that\[
 B\left(\begin{pmatrix}v\\f\end{pmatrix},\begin{pmatrix}v\\ f \end{pmatrix}\right)\geq \epsilon\left\|\begin{pmatrix}v\\f\end{pmatrix}\right\|^2_{H^1(\Tt)\times H^1(\Tt)},
\]
which proves the coercivity of \(B\).
  
Because $m$, $u$, and $H$ are $C^{2, \frac 1 2}$-functions on the compact set $[0,1]$,  we have that $m$, $u$, $m_x$, $u_x$, $u_{xx}$, $H$, $H'(u_x)$, and $H''(u_x)$ are bounded. Therefore, there exists
a positive constant, \(C\), that depends only on
these bounds and for which
\begin{equation*}
  \left|B\left(\begin{pmatrix}v\\f\end{pmatrix},\begin{pmatrix}w_1\\ w_2 \end{pmatrix}\right)\right|\leq C{\left\|\begin{pmatrix}v\\
f\end{pmatrix}\right\|}_{H^1(\Tt)\times
H^1(\Tt)}{\left\|\begin{pmatrix}w_1\\w_2\end{pmatrix}\right\|}_{H^1(\Tt)\times
H^1(\Tt)},
\end{equation*}
where we also used H\"older's inequality. 
This proves the boundedness of \(B\).
  
Finally, we fix $b=(b_1,b_2)\in C^{0, \frac 1 2}(\Tt)\times C^{0, \frac 1 2}(\Tt)$, 
  and we consider the  bounded and linear functional
\(G: H^1(\Tt)\times
H^1(\Tt) \to \Rr\) defined for \((w_1,
w_2)\in H^1(\Tt)\times
H^1(\Tt)\) by
  \[
  G\begin{pmatrix}w_1\\w_2\end{pmatrix}=\int_0^1
  \left( -b_1w_2+b_2w_1 \right) dx.\]
By the Lax-Milgram theorem, there exists a unique $(v,f)\in H^1(\Tt)\times H^1(\Tt)$ such that for
all \((w_1,
w_2)\in H^1(\Tt)\times
H^1(\Tt)\), we have   
\begin{equation*}
B\left(\begin{pmatrix}v\\f\end{pmatrix},
    \begin{pmatrix}w_1\\w_2\end{pmatrix}\right)=
    G\begin{pmatrix}w_1\\w_2\end{pmatrix}.
\end{equation*}
This is equivalent to saying that  for
all \((w_1,
w_2)\in H^1(\Tt)\times
H^1(\Tt)\),
\begin{equation*}
 B\left(\begin{pmatrix}v\\f\end{pmatrix}, 
  \begin{pmatrix}-w_2\\w_1\end{pmatrix}\right)=
  G\begin{pmatrix}-w_2\\w_1\end{pmatrix} =\int_0^1
  \left( -b_1w_1-b_2w_2 \right) dx .
\end{equation*}
  From this and \eqref{Bvfreg}, we conclude that $L(f,v)=b$ has a  unique weak solution $(f,v)\in H^1(\Tt)\times H^1(\Tt)$. Because $b\in  C^{0, \frac 1 2}(\Tt)\times
C^{0, \frac 1 2}(\Tt)$ is arbitrary, $L$ is injective. 
To prove surjectivity, it suffices to check that   
  the weak  solution of 
$L(f,v)=b$  is in $C^{2,\frac 1 2}(\Tt)\times C^{2,\frac 1 2}(\Tt)$.
This higher regularity  follows from a bootstrap
argument.
    
    Fix  $b=(b_1,b_2)\in C^{0, \frac 1
2}(\Tt)\times C^{0, \frac 1 2}(\Tt)$ and let  $(f,v)\in
H^1(\Tt)\times H^1(\Tt)$ be the weak solution
of $L(f,v)=b$    given by the Lax-Milgram theorem. Then, we have the following identity in the weak sense:
\begin{equation}
\label{eqnvxx}
   v_{xx}=\frac{g}{1+\epsilon^2+\epsilon H''(u_x)m},
\end{equation}
  where $g=v(1+\epsilon^2)+H'(u_x)v_x-\alpha m^{\alpha-1}f-\epsilon v_x(H'(u_x)m)_x-\epsilon(H'(u_x)f)_x-\epsilon b_2-b_1\in L^2(\Tt)$. We recall that \(1+\epsilon^2+\epsilon H''(u_x)m >1\).  Hence, $v_{xx}\in L^2(\Tt)$, and so $v\in H^2(\Tt)$. Moreover, because %%
\begin{equation}
\label{eqnfxx}
f_{xx}=f-(H''(u_x)v_xm)_x-(H'(u_x)f)_x+\epsilon(v-v_{xx})-b_2
\end{equation}
in the weak sense,  similar arguments yield  $f_{xx}\in L^2(\Tt)$ and $f\in H^2(\Tt)$.
  
So far, \((f,v) \in C^{1,
\frac 1 2}(\Tt)\times
C^{1, \frac 1 2}(\Tt) \). This implies that \(g\in
C^{0, \frac1 2}(\Tt)\). Then, using the fact that
\(1+\epsilon^2+\epsilon
H''(u_x)m \) also belongs to \(C^{0, \frac1 2}(\Tt)\) and is bounded from below by 1, from
\eqref{eqnvxx} it follows that \(v_{xx} \in C^{0, \frac1
2}(\Tt) \). Consequently, in view of \eqref{eqnfxx},
  \(f_{xx} \in C^{0,
\frac1
2}(\Tt) \). Hence, \((f,v) \in C^{2,
\frac 1 2}(\Tt)\times
C^{2, \frac 1 2}(\Tt) \).
Therefore, the unique solution given by the Lax-Milgram theorem is a strong solution with $C^{2,\frac 1 2}$ regularity. Thus, $L$ is surjective. 
Because $L$  is injective and surjective, it is an isomorphism. 
\end{proof}

\section{Proof of the Main Theorem}
\label{pthm}
%WE NEED ASSUMPTIONS OF: implicit function theorem, uniform bounds of u and ux, m,mx etc., equicontinuity of u,m,ux,etc, H''>0 (H in C2 & convex), V at least continuous, alfa>0,..?

In this last section, we prove Theorem \ref{T1}.
We assume that    \(\epsilon>0\) satisfies
\(\epsilon < \min\{1,\epsilon_0,\bar
\epsilon_0\} \), where \(\epsilon_0\) and \(\bar
\epsilon_0\)  are given
by Proposition~\ref{prop_1} and  Lemma~\ref{lowboum_lemma},
respectively. 

Let $F$ be the functional  defined in \eqref{F}. 
For each \(\lambda \in [0,1]\),  consider the problem of finding \((u,m) \in C^{2, \frac 1 2}(\Tt)\times
C^{2,\frac 1 2}(\Tt;]0,\infty[) \) satisfying
\eqref{F=0}. 
From 
Propositions~\ref{prop_1} and \ref{hrofsol},
such a pair \((u,m)\) exists
for $\lambda = 0$. 
Next, using the continuation
method, we prove that this is true not only for
\(\lambda=0\) but also
for all  $\lambda \in [0,1]$. 
 
More precisely, let 
$\Lambda$ be the set of values, $\lambda\in [0,1]$, for which  equation \eqref{F=0} has
a solution $(u,m)\in C^{2, \frac 1 2}(\Tt)\times C^{2,\frac 1 2}(\Tt)$
with $m\geq\bar m$ in \(\Tt\), where $\bar m>0$ is given by 
Proposition \ref{lboundsprop}. Note that \(\bar
m\) does not depend on \(\lambda\) (see Remark~\ref{indepV2}).
As we just argued,  
 $\Lambda$ is a non-empty set. In the subsequent two
propositions, we show that   
 $\Lambda$ is a closed and open subset of $[0,1]$.
  Consequently, $\Lambda=[0,1]$.

\begin{proposition}
\label{prop121}
Suppose that Assumptions~\ref{ass1}--\ref{ass6} hold. 
Then, 
$\Lambda$ 
is a closed subset of $[0,1]$.
\end{proposition}

\begin{proof}
Let $(\lambda^n)_{n\in \mathbb{N}}
\subset \Lambda$ and \(\lambda
\in [0,1]\) be such that $\lim_{n\to \infty}\lambda^n = \lambda
$. We claim that $\lambda
\in \Lambda$.

By definition of $ \Lambda$, for each \(n\in\Nn\), there exists
 $(u^n,m^n) \in C^{2,\frac 1 2} (\Tt) \times C^{2,\frac 1 2}(\Tt)$
satisfying \eqref{F=0} and $m^n\geq \bar m$ 
in \(\Tt\).
Then, by Proposition~\ref{hbounds} (also see Remark~\ref{indepV}),
$( u^n )_{n \in \mathbb{N} }$,  $(
m^n )_{n \in \mathbb{N} }$, $( u_x^n )_{n
\in \mathbb{N} }$, and $ ( m_x^n )_{n
\in \mathbb{N} }$ are uniformly bounded in \(C^{0,\frac
1 2}(\Tt)\). Consequently, by the Arzel\'a-Ascoli theorem, we can find \((u,m,\tilde u, \tilde m)
\in C^{0,\frac
1 2}(\Tt) \times C^{0,\frac
1 2}(\Tt) \times C^{0,\frac
1 2}(\Tt) \times C^{0,\frac
1 2}(\Tt)\) such that, up to a subsequence that
we do not relabel,
\begin{equation}
\label{convunif}
\lim_{n\to \infty} \Vert (u^n,m^n, u^n_x,  m^n_x) - (u,m,\tilde u, \tilde m)\Vert_\infty =0.
\end{equation}
We now recall that if 
$ (w^n)_{n\in\Nn}$ is a sequence of differentiable functions
on $ [0, 1]$ such that $ (w^n)_{n\in\Nn}$ converges uniformly
to some $w$ on $[0,1]$ and such that  $ (w^n_x)_{n\in\Nn}$ converges uniformly
on $[0,1]$, then $ w_x = \lim_{n\to \infty} w_x^n
$ on \([0,1]\). Consequently, by \eqref{convunif},
we have that \(\tilde u = u_x\)
and \(\tilde m = m_x\).

Next, we show that $(u_{xx}^n)_{n\in \mathbb{N}}$ and $(m_{xx}^n)_{n\in
\mathbb{N}}$ are also uniformly convergent sequences
on \([0,1]\). In view of \eqref{ABC123}, we have that, for every $n \in
\mathbb{N}$,
\begin{equation} \label{uc_uxx}
u^n_{xx}=\frac{(1+\epsilon^2)u^n+H(u^n_x)-\epsilon
+\lambda^n V(x)-(m^n)^{\alpha}-\epsilon H'(u^n_x)m^n_x}{1+\epsilon^2+\epsilon
H''(u^n_x)m^n}.
\end{equation}
By Assumption~\ref{ass3.5} and by the uniform convergence of $\big(u^n,m^n,\lambda^n,
u_x^n, m_x^n\big)_{n\in \mathbb{N}}$ to $\big(u,
m, \lambda, u_x, m_x\big)$ on \([0,1]\), it follows
from \eqref{uc_uxx}  that $(u_{xx}^n )_{n\in
\mathbb{N}} $ converges uniformly on \([0,1]\).
Then,  the limit of $(u_{xx}^n )_{n\in
\mathbb{N}} $  is necessarily \(u_{xx}\).
Analogous arguments (see \eqref{7-2})  give that $(m_{xx}^n )_{n\in
\mathbb{N}} $ converges uniformly to \(m_{xx}\)
on \([0,1] \). Consequently, \((u,m) \in C^{2, \frac 1 2}(\Tt)\times
C^{2,\frac 1 2}(\Tt;]0,\infty[).\) 
Moreover, $\lim_{n \to \infty} F(u^n,m^n,\lambda^n)
 =  F(u,m,\lambda) $. 
 Finally, because for all \(n\in\Nn\), \(F(u^n,m^n,\lambda^n)=0\)
 and  $m^n\geq \bar m$ 
in \(\Tt\), we have that \(F(u,m,\lambda)=0\)
 and  $m\geq \bar m$
in \(\Tt\). Thus,   $\lambda
\in \Lambda$. This
  completes the proof.
\end{proof}

  \begin{proposition}
        \label{Pr2}
Suppose that Assumptions~\ref{ass1}--\ref{ass6} hold. Then, $\Lambda$ 
is an open subset of $[0,1]$.
  \end{proposition}  

  \begin{proof}
  Let $\lambda_0 \in \Lambda$. Then, there exists $(u_0,m_0)\in C^{2,\frac 1 2 }(\Tt)\times C^{2,\frac 1 2 }(\Tt)$ satisfying  $F(u_0,m_0,\lambda_0) = 0$ 
and $m_0\geq \bar m$ in \(\Tt\).
By Proposition~\ref{Lisisom} and by the implicit function theorem in Banach spaces (see,  for example, \cite{D1}), we can find $\delta > 0$ such that, for every $\lambda^* \in\, ]\lambda-\lambda_0, \lambda+
\lambda_0[$, there exists $(u^*,m^*)\in C^{2,\frac 1 2 }(\Tt)\times C^{2,\frac 1 2 }(\Tt)$
satisfying $F(u^*,m^*,\lambda^*) = 0$ and  $ m^* \geq \bar m$ in $\Tt$.
Moreover, the implicit function theorem also guarantees that the map 
$\lambda^*\mapsto m^*$ is continuous. Hence, if $\delta$ is small enough, 
we have $m^*>0$ in $\Tt$. Then, Proposition~\ref{lowboum_lemma} gives $m^*>\bar m$ in $\Tt$. 
Therefore, $\lambda^* \in \Lambda$ and, consequently, $\Lambda$ is open.
  \end{proof}

Finally, we sum up the proof of our main result.
\begin{proof}[Proof of Theorem~\ref{T1}]
Let \(\epsilon>0\) be such that \(\epsilon < \min\{1,\epsilon_0,\bar
\epsilon_0\} \), where \(\epsilon_0\) is given by Proposition~\ref{prop_1} and where \(\bar
\epsilon_0\) is given
by Lemma~\ref{lowboum_lemma}.

Propositions~\ref{prop121} and \ref{Pr2} give that $\Lambda$ is a relatively open and closed set in $[0,1]$.
It is a  non-empty set due to Propositions~\ref{prop_1},
\ref{hrofsol}, and  \ref{lboundsprop}.
Hence, $\Lambda=[0,1]$. Finally, we observe that
Theorem~\ref{T1} corresponds to the  $\lambda=1$
case.
\end{proof}  
  
\begin{section}*{Acknowledgements}
This summer camp would not have been possible without
major help and support from KAUST.  David Yeh
and his team at the Visiting Student Research
Program
did a fantastic job in organizing all the logistics.
In addition, the CEMSE Division provided valuable
additional support. Finally, we would like to
thank the faculty, research scientists, post-docs
and Ph.D. students who, in the first three weeks
of the semester, made time to give courses, lectures,
and work with the students.  
\end{section}

  \bibliographystyle{alpha}
  
  \bibliography{mfg}

\begin{thebibliography}{10}

\bibitem{cgbt}
P.~Cardaliaguet, P.~Garber, A.~Porretta, and D.~Tonon.
\newblock Second order mean field games with degenerate diffusion and local
  coupling.
\newblock {\em Preprint}, 2014.

\bibitem{MR3333058}
M.~Cirant.
\newblock Multi-population {M}ean {F}ield {G}ames systems with {N}eumann
  boundary conditions.
\newblock {\em J. Math. Pures Appl. (9)}, 103(5):1294--1315, 2015.

\bibitem{D1}
J.~Dieudonn{\'e}.
\newblock {\em Foundations of modern analysis}.
\newblock Vol. I, Academic Press, New York, 1969.

\bibitem{E6}
L.~C. Evans.
\newblock {\em Partial Differential Equations}.
\newblock Graduate Studies in Mathematics. American Mathematical Society, 1998.

\bibitem{FG2}
R.~Ferreira and D.~Gomes.
\newblock Existence of weak solutions for stationary mean-field games through
  weak solutions.
\newblock {\em Preprint}.

\bibitem{GMit}
D.~Gomes and H.~Mitake.
\newblock Stationary mean-field games with congestion and quadratic
  {H}amiltonians.
\newblock {\em Preprint}.

\bibitem{GPat}
D.~Gomes and S.~Patrizi.
\newblock Obstacle mean-field game problem.
\newblock {\em To appear in Interfaces and Free Boundaries}, 2013.

\bibitem{GPatVrt}
D.~Gomes, S.~Patrizi, and V.~Voskanyan.
\newblock On the existence of classical solutions for stationary extended mean
  field games.
\newblock {\em Nonlinear Anal.}, 99:49--79, 2014.

\bibitem{GPim1}
D.~Gomes and E.~Pimentel.
\newblock Local regularity for mean-field games in the whole space.
\newblock {\em Preprint}.

\bibitem{GPim2}
D.~Gomes and E.~Pimentel.
\newblock Time dependent mean-field games with logarithmic nonlinearities.
\newblock {\em Preprint}.

\bibitem{GPM3}
D.~Gomes, E.~Pimentel, and H~Sanchez-Morgado.
\newblock Time dependent mean-field games in the superquadratic case.
\newblock {\em Preprint}, 2013.

\bibitem{GPM2}
D.~Gomes, E.~A. Pimentel, and H.~S{\'a}nchez-Morgado.
\newblock Time-dependent mean-field games in the subquadratic case.
\newblock {\em Comm. Partial Differential Equations}, 40(1):40--76, 2015.

\bibitem{GPM1}
D.~Gomes, G.~E. Pires, and H.~S{\'a}nchez-Morgado.
\newblock A-priori estimates for stationary mean-field games.
\newblock {\em Netw. Heterog. Media}, 7(2):303--314, 2012.

\bibitem{GR}
D.~Gomes and R.~Ribeiro.
\newblock Mean field games with logistic population dynamics.
\newblock {\em 52nd IEEE Conference on Decision and Control (Florence, December
  2013)}, 2013.

\bibitem{GM}
D.~Gomes and H.~S{\'a}nchez~Morgado.
\newblock A stochastic {E}vans-{A}ronsson problem.
\newblock {\em Trans. Amer. Math. Soc.}, 366(2):903--929, 2014.

\bibitem{GS}
D.~Gomes and J.~Sa{\'u}de.
\newblock Mean field games models---a brief survey.
\newblock {\em Dyn. Games Appl.}, 4(2):110--154, 2014.

\bibitem{Graber2}
J.~Graber.
\newblock Weak solutions for mean field games with congestion.
\newblock {\em Preprint}.

\bibitem{Caines2}
M.~Huang, P.~E. Caines, and R.~P. Malham{\'e}.
\newblock Large-population cost-coupled {LQG} problems with nonuniform agents:
  individual-mass behavior and decentralized {$\epsilon$}-{N}ash equilibria.
\newblock {\em IEEE Trans. Automat. Control}, 52(9):1560--1571, 2007.

\bibitem{Caines1}
M.~Huang, R.~P. Malham{\'e}, and P.~E. Caines.
\newblock Large population stochastic dynamic games: closed-loop
  {M}c{K}ean-{V}lasov systems and the {N}ash certainty equivalence principle.
\newblock {\em Commun. Inf. Syst.}, 6(3):221--251, 2006.

\bibitem{ll1}
J.-M. Lasry and P.-L. Lions.
\newblock Jeux \`a champ moyen. {I}. {L}e cas stationnaire.
\newblock {\em C. R. Math. Acad. Sci. Paris}, 343(9):619--625, 2006.

\bibitem{ll2}
J.-M. Lasry and P.-L. Lions.
\newblock Jeux \`a champ moyen. {II}. {H}orizon fini et contr\^ole optimal.
\newblock {\em C. R. Math. Acad. Sci. Paris}, 343(10):679--684, 2006.

\bibitem{ll3}
J.-M. Lasry and P.-L. Lions.
\newblock Mean field games.
\newblock {\em Jpn. J. Math.}, 2(1):229--260, 2007.

\bibitem{LCDF}
P.-L. Lions.
\newblock College de {F}rance course on mean-field games.
\newblock 2007-2011.

\bibitem{LIMA}
P.-L. Lions.
\newblock {I}{M}{A}, {U}niversity of {M}inessota. {C}ourse on mean-field games.
  {V}ideo. http://www.ima.umn.edu/2012-2013/sw11.12-13.12/.
\newblock 2012.

\bibitem{PV15}
E.~Pimentel and V.~Voskanyan.
\newblock Regularity for second-order stationaty mean-field games.
\newblock {\em Preprint}.

\bibitem{porretta}
A.~Porretta.
\newblock On the planning problem for the mean-field games system.
\newblock {\em Dyn. Games Appl.}, 2013.

\bibitem{porretta2}
A.~Porretta.
\newblock Weak {S}olutions to {F}okker--{P}lanck {E}quations and {M}ean {F}ield
  {G}ames.
\newblock {\em Arch. Ration. Mech. Anal.}, 216(1):1--62, 2015.

\end{thebibliography}

\end{document}